\documentclass[12pt]{amsart}
\usepackage[english]{babel}
\usepackage{verbatim,amsfonts,amssymb,amsmath,graphicx}
\usepackage{amsmath}
\usepackage{amsthm}
\usepackage{amsfonts}
\usepackage{amssymb}
\usepackage{graphicx}
\usepackage{latexsym}
 \usepackage{lipsum}% http://ctan.org/pkg/lipsum
\makeatletter
\g@addto@macro{\endabstract}{\@setabstract}
\newcommand{\authorfootnotes}{\renewcommand\thefootnote{\@fnsymbol\c@footnote}}%
\makeatother
\usepackage{amsmath,amsthm,amsfonts,amssymb}

\usepackage{epsfig}

\usepackage{amssymb}
\usepackage[all]{xy}
\xyoption{poly}
\usepackage{fancyhdr}
\usepackage{wrapfig}

\usepackage[left=1.5cm, right=1.5cm, top = 1.5cm, bottom=2 cm]{geometry}
\newcommand{\mathsym}[1]{{}}
\newcommand{\unicode}[1]{{}}
\newtheorem{thm}{Theorem}
\newtheorem*{theorem*}{Theorem}
\newtheorem{lem}{Lemma}

\newtheorem{prop}{Proposition}

\theoremstyle{remark}

\newtheorem{defn}{Definition}

\def\<{\langle}
\def\>{\rangle}

\def\0b{\mathbf{0}}
\def\M{\mathcal{M}}

\def\Z{\mathbb{Z}}
\def\id{\mathrm{id}}

\def \go{||b_1-b_2||^{\alpha}}
\def\Me{\mathcal{M}_{\varepsilon}}

\def \supe {\sup_{x,y \in [0,\varepsilon]}}

\def \Hae  {\mathcal{H}^{\alpha}_{\varepsilon}}
\def \lb {\lambda_b}

\def \Td {\mathbb{T}^d}
\def\Lip{\mathrm{Lip}}
\def\<{\langle}
\def\>{\rangle}
\title[Sternberg linearization theorem for skew products]
{Sternberg linearization theorem for skew products}

\begin{document}
\maketitle
\begin{center}
\authorfootnotes
YULIJ ILYASHENKO \footnote{email: yulij@math.cornell.edu}\textsuperscript{1,2},
OLGA ROMASKEVICH\footnote{email:olga.romaskevich@ens-lyon.fr}\textsuperscript{1,3} \par \bigskip
 \textsuperscript{1} National Research University Higher School of Economics, Moscow \\
  \textsuperscript{2} Cornell University
  \textsuperscript{3} \'{E}cole Normale Sup\'{e}rieure de Lyon
  \end{center}

\begin{abstract}
We present a new kind of normalization theorem: linearization theorem for skew products. The normal form is a skew product again, with the fiber maps linear. It appears, that even in the smooth case, the conjugacy is only H\"older continuous with respect to the base. The normalization theorem mentioned above may be applied to perturbations of skew products and to the study of new persistent properties of attractors.
\end{abstract}

%\begin{flushright}
%\emph{The course of true love never did run smooth.\\
%Midsummer Night's Dream\\
%William Shakespeare
%}
%\end{flushright}
\section{Setting and statements}

\subsection{Motivation}\label{section_motivation}

This paper is devoted to a normalization theorem for H\"older skew products. We begin with the motivation for the choice of this class of maps.

According to a heuristic principle going back to \cite{program}, generic phenomena that occur in random dynamical systems on a compact manifold may also occur for diffeomorphisms of manifolds of higher dimensions. Random dynamical systems are equivalent to skew product homeomorhisms over Bernoulli shift. Some new effects found for these homeomorphisms were transported later to skew product diffeomorphisms over hyperbolic maps with compact fibers. These diffeomorphisms are in no way generic. Their small perturbations are skew products again, whose fiber maps are smooth but only continuous with respect to the base point \cite{HPS}.

Recently it was discovered that these fiber maps are in fact H\"older with respect to the base point \cite{G, IN, Revisited}.

New effects found for skew product diffeomorphisms are then transported to H\"older skew products, and thus proved to be generic. This program was carried on in \cite{four, Skew, Thick, Intermingled, kv, urkud}.

This motivates the study of H\"older skew products. We now pass to our main results.

\subsection{Main statements}
Consider a skew product diffeomorphism over an Anosov map in the base with the fiber a segment. In more detail, let $M=\mathbb{T}^d \times I$, $\mathbb{T}^d$ is a a $d$-dimensional torus, $I=[0,1]$. Consider a boundary preserving skew product

\begin{equation}\label{eq:new:def_skew}
F: M \rightarrow M, (b,x) \mapsto \left( Ab, f_b(x) \right),
\end{equation} where $f_b (0)=0, f_b(1)=1$, the fiber map $f_b$ is an orientation preserving diffeomorphism $I \rightarrow I$ and the base map $A$ is a linear hyperbolic automorphism of a torus.

Suppose also that $f$ is H\"older continuous in $x$ with respect to the $C^k$-norm; i.e. that there exist constants $C_k, \beta>0$ such that for any $b, b' \in \mathbb{T}^d$ the following holds:

\begin{equation}\label{eq:new:Holder_f}
||f_b-f_{b'}||_{C^k} \leq C_k ||b-b'||^{\beta}
\end{equation}

This assumption appears in slightly different settings as a statement in a number of articles on partial hyperbolicity: for exmple, in  \cite{Revisited} the estimate \eqref{eq:new:Holder_f} is true for $k=0$, in  \cite{horseshoe} for $k=1$ and in \cite{G} for any $k$. Now we will state the main results that we are proving in a hope to apply them to the study of the skew products: for example, to drastically simplify the proofs in \cite{Thick}.

\begin{thm}\label{thm:new:1}
Consider a map $F$ of the form \eqref{eq:new:def_skew} with the property \eqref{eq:new:Holder_f} for some fixed $k$ and $C^k$ -- smooth on the $x$ variable,  $k \geq 2$. Let $O \in \mathbb{T}^d \times \{ 0\}$ be its hyperbolic fixed point. Then there exists a neighborhood $U$ of $O$ and a fiber preserving homeomorphism
\begin{equation}\label{eq_conjugacy_form}
H: (U, O) \rightarrow U,\;\; (b,x) \mapsto (b, x+h_b(x)) \;\; h_b(0)=\frac{\partial h_b} {\partial x}(0)=0,
\end{equation}
 such that
\begin{itemize}
\item[1.] $H$ conjugates $F$ in $(U, O)$ with its "fiberwize linearization"
\begin{equation}\label{eq_linearization_form}
F_0: (b,x) \mapsto (Ab, \lambda_b x),
\end{equation}
where \begin{equation}
\lambda_b=f_b'(0).
\end{equation}
This means that
\begin{equation}\label{eq_equation_conjugacy}
F \circ H = H \circ F_0.
\end{equation}

\item[2.] $H$ is smooth on $x$ for $b$ fixed: the degree of smoothness is $k-2$
\item[3.] $H$ is fiberwise H\"older: there exist constants $\tilde{C}_l, \alpha>0$ such that for any $l, 0 \leq l \leq k-2$ holds
\begin{equation}\label{eq:new:h4}
||h_b-h_{b'}||_{C^l} \leq \tilde{C}_l ||b-b'||^{\alpha}
\end{equation}
 such that
\begin{equation}\label{eq:new:alpha}
\alpha \in \min (\beta, \log_{\mu} q),
\end{equation}
where $\mu$ is the largest magnitude of  eigenvalues of $A$.\textbf{•}
\end{itemize}
\end{thm}

This theorem is local: the conjugacy relation $F \circ H = H \circ F_0$ holds in a neighborhood of $O$ only. We will reduce this theorem to the following two results.

\begin{thm}\label{thm:new:2}
Consider the same $F$ as in Theorem \ref{thm:new:1}, with $k=2$. Let in addition
\begin{equation*}\label{eq:lambda_bound}
\lambda_b \leq q < 1 \; \;\;\;\forall b \in \mathbb{T}^d
\end{equation*}
Then the map $H$ with the properties 1, 2, 3 mentioned in Theorem \ref{thm:new:1} is defined in a set
\begin{equation*}
M_{\varepsilon}=\left\{
(b,x) \in \mathbb{T}^d \times [0,1] \left|\right. x \in [0, \varepsilon]
\right\}
\end{equation*}
for some $\varepsilon>0$ and is continuous on this set. Moreover, for $l=0$, the relation \eqref{eq:new:h4} holds for any $\alpha$ as in \eqref{eq:new:alpha}.
\end{thm}

\begin{thm}\label{thm:new:3}
Suppose that all assumptions of Theorem \ref{thm:new:2} hold, except for $k$ is arbitrary now. Then the map $H$ with the same properties as in Theorem \ref{thm:new:2} exists. Moreover, $H$ is $C^{k-2} $ fiberwise smooth, and satisfies the H\"older condition \eqref{eq:new:alpha}
for $l = k - 2$.
\end{thm}

Theorems \ref{thm:new:2} and \ref{thm:new:3} are the main results of the paper. They are similar: the first one  claims that the fiberwise conjugacy $H$ is continuous in the $C$-norm with respect to  the point of the fiber, and the second one improves this result -- by decreasing the neighborhood in the fiber, and replacing the $C$-norm by the $C^l$ one. The main part of the paper to follow is the proof of Theorem \ref{thm:new:2}. At the end we present a part of the proof of Theorem \ref{thm:new:3}. Namely, we prove that the maps $h_b$ are $(k-2)$-smooth , but we do not prove that the derivatives $
\frac{\partial^j h_b}{\partial x^j},  1 \leq j \leq k-2
$ are H\"older in $b$. This may be proven in the same way as the H\"older property for $h_b$, but requires more technical details that we skip here.

At the beginning of the next section we
deduce  Theorem \ref{thm:new:1}  from Theorems \ref{thm:new:2}, \ref{thm:new:3}.

\subsection{Comparison with the theory of ``nonstationary normal forms''}

This theory is  developed in \cite{GK} and \cite{Gu}. The closer to our results is Theorem 1.2 of \cite{GK} proved in full detail as Theorem 1 of \cite{Gu}. For future references we will call it the GK-Theorem. This theorem considers a wide class of maps to which the map \eqref{eq:new:def_skew} belongs, provided that it satisfies the so called \emph{narrow band spectrum condition.} In assumptions of our Theorem \ref{thm:new:1}, this condition has the form:

\begin{equation}\label{eqn:narrow}
  (\max_{b \in \mathbb{T}^d} \lambda_b)^2 < \min_{b \in  \mathbb{T}^d}  \lambda_b.
\end{equation}

This is a restrictive condition, not required in Theorem \ref{thm:new:2}. Moreover, GK-Theorem claims properties 1 and 2 of the map $H$, that is, $H$ is a fiberwise smooth  topological conjugacy, and does not claim H\"{o}lder continuity of the fiberwise maps of $H$, see \eqref{eq:new:alpha}. To summarize, Theorems \ref{thm:new:2}, \ref{thm:new:3} improve the GK-Theorem for the particular class of maps \eqref{eq:new:def_skew}, skipping the narrow band spectrum condition and adding Property 3, the H\"{o}lder continuity.

Statements 1 and 2 of Theorem \ref{thm:new:1} may be deduced from the GK-Theorem. Indeed, any continuous function $\lambda < 1 $ satisfies condition \eqref{eqn:narrow} in a suitable neighborhood of any point. On the other hand, Theorem \ref{thm:new:1} in its full extent is easily deduced from Theorems \ref{thm:new:2}, \ref{thm:new:3}.

\section{The plan of the proof}\label{section_plan}
\subsection{Globalization}
 Theorem \ref{thm:new:1} is proved with
  the help of standard globalization technics (see, for instance, \cite{Nonlocal}).
Without loss of generality, hyperbolicity of the skew-product $F$ implies that $\lambda_b<1$ in the neighborhood $U \in \mathbb{T}^d$ of the fixed point O. If not, we pass to the inverse mapping $F^{-1}$. Let $K$ be a compact subset of $U$. Let us take a smooth cut function $\varphi: \mathbb{T}^d \rightarrow [0,1]$ such that

\begin{equation}
\varphi |_K \equiv 0, \varphi | _{\mathbb{T}^d \setminus U} \equiv 1
\end{equation}

Instead of the initial function $f_b(x)$ on the fibers let us consider the function
\begin{equation}
\tilde{f}_b(x)=f_b(x)(1-\varphi)+\frac{x}{2}\varphi
\end{equation}

Then a map

\begin{equation}
\tilde{F}: X_{\varepsilon} \rightarrow X, (b,x) \mapsto (Ab, \tilde{f}_b(x))
\end{equation}
has the following list of properties:

\begin{itemize}
\item[1.] $\tilde{F}$ coincides with $F$ in the neighborhood of the fixed point $O$
\item[2.] $\tilde{F}$ is attracting near the zero layer: if $\tilde{\lambda}_b:=\tilde{f}_b'(0)$ then
\begin{equation}\label{eq_additional}
\tilde{\lambda}_b<1 \;\; \forall b \in \mathbb{T}^d
\end{equation}
\end{itemize}

So without loss of generality we may assume from the very beginning that $\lambda_b \in (0,1)$ everywhere on the base $\mathbb{T}^d$. Moreover, the conjugacy $H$ is to be found in the whole $M_{\varepsilon}$: in other words, the equality $F \circ H = H \circ F_0$ will hold on the full neigbourhood $M_{\varepsilon}$ of the base.

All the rest of the article deals with the proof of the global result, i.e. Theorem \ref{thm:new:2} and Theorem \ref{thm:new:3}.

Now let us prove Theorem \ref{thm:new:2}: the main idea is to use fixed point theorem to prove the existence of the conjugacy $H$: we should just properly define the functional space and a contraction operator in it. In the following sections we will do all of it, postponing some calculcations as well as the proof of Theorem \ref{thm:new:3} to the Appendix (Section \ref{section_appendix}).

\subsection{Homological and functional equations}
Suppose that \begin{equation}\label{eq_f_expansion_quadratic}
f_b(x)=\lambda_b x+R_b(x), \;\; R_b(x)=O(x^2), x \rightarrow 0.
\end{equation}
Then finding the conjugacy $H: M_{\varepsilon} \rightarrow M_{\varepsilon}$ of the form \eqref{eq_conjugacy_form} satisfying \eqref{eq_equation_conjugacy}  is equivalent to finding the solution $\bar{h}_b(x)$ of a so-called \emph{functional equation}
\begin{equation}\label{eq_functional_equation_long}
\bar{h}_{Ab}\left(\lambda_b x\right)-\lambda_b \bar{h}_b(x)=R_b\left(x+\bar{h}_b(x)\right)
\end{equation}
More briefly, this equation can be written in a\emph{ compositional form} as
\begin{equation}\label{eq_functional_equation_short}
\bar{h}\circ F_0-\lambda \bar{h} = R \circ H
\end{equation}
Here we denote by $\lambda$ the operator of multiplication by  $\lambda_b$, here $b$ is an argument of the considered function.

In the following we will be working not with the quadratic part of the conjugacy itself but with this part divided by $x^2$. That's why we change the notations in such a way: we write bars for the functions in the space of quadratic parts for possible conjugacy maps and we don't write bars for the same functions divided by $x^2$, for instance, $h_b(x):=\frac{\bar{h}_b(x)}{x^2}$.  In a similar fashion, $Q_b(x):=\frac{R_b(x)}{x^2}$.

The functional equation is a hard one to solve since the function $\bar{h}_b(x)$ is present in both sides of the equation. One may simplify functional equation and consider a gentler form of the equation on $\bar{h}_b(x)$, a \emph{homological equation:}
\begin{equation}\label{eq_homological_equation_short}
\bar{h} \circ F_0 - \lambda \bar{h} = R
\end{equation}

The solution of the homological equation doesn't give the conjugacy but is a useful tool in the investigation. Homological eqation can be rewritten equivalently in terms of $h$ and $Q$ as
\begin{equation}\label{eq_homological_equation_Q}
\lambda^2 h \circ F_0 - \lambda h = Q
\end{equation}

\subsection{Operator approach}\label{subsection_operator_approach}
Let us consider a space $\mathcal{M}$ of real-valued functions defined on $M$ which are continuous on $b \in \mathbb{T}^d$ and smooth in $x \in [0,1]$:

\begin{equation}
\mathcal{M}:=\left\{
\bar{h}_b(x) \in M \left|\right.
\bar{h}_{\cdot}(x) \in C(\mathbb{T}^d), \bar{h}_b (\cdot) \in C^k [0,1]
\right\}
\end{equation}
 Let us define an operator $\bar{\Psi}: \mathcal{M} \rightarrow \mathcal{M}$ on it which acts on a function $h(b,x)$ by associating to it the left-hand side of the homological equation \eqref{eq_homological_equation_short}.
With the use of this operator, equation \eqref{eq_homological_equation_short} can be rewritten in a form $\bar{\Psi} \bar{h} = R$.

Denote $\bar{L}$ an inverse operator to $\bar{\Psi}$. The operator $\bar{L}$ is \emph{solving} homological equation: if the right-hand side is $R$ then $\bar{L}R=\bar{h}$ and $\bar{L} \bar{\Psi} = \id$. From now on, operator $\bar{L}$ will be referred as \emph{homological operator}.

Let us define a \emph{shift operator} $\bar{\Phi}: \mathcal{M} \rightarrow \mathcal{M}$ which acts as

\begin{equation}\label{eq_shift_operator_def}
\bar{\Phi} \bar{h} (b,x) = R_b\left(x + \bar{h}_b(x)\right)
\end{equation}

Then the functional equation on the function $\bar{h}$ can be rewritten in the form $\bar{\Psi} \bar{h} = \bar{\Phi} \bar{h}$ or, equivalently, $\bar{h} = \bar{L} \bar{\Phi} \bar{h}$. So the search for conjugacy is \textbf{equivalent} to the search of a fixed point for the operator $\bar{L} \bar{\Phi}$ in the space $\mathcal{M}$.

Let us note for the future that the operator $\bar{\Psi}$ (as well as its inverse $\bar{L}$) is a linear operator on the space of formal series although operator $\bar{\Phi}$ is not at all linear: for instance, it sends a zero function to $R_b(x)$.

\subsection{Choice of a functional space for the Banach fixed point theorem}
We will be using the simplest of the forms of a contraction mapping principle by considering a contracting mapping defined on a metric space $\mathcal{N}$ and preserving its closed subspace $N$.

Let us define $\mathcal{N}, d$ and $N$ for our problem: a contraction mapping $f$ will be a slight modification of the composition of operators $\bar{L} \bar{\Phi}$ considered in Section \ref{subsection_operator_approach}. Note that operator $\bar{L}$ as well as operator $\bar{\Phi}$ preserve the subspace of $\mathcal{M}$ of functions starting with quadratic terms in $x$ which we will denote by $\mathcal{M}^2$:

\begin{equation}\label{def_space_M_two}
\mathcal{M}^2=\left\{
\bar{h}_b(x) \in \mathcal{M} | \bar{h}_b(x)=x^2 h_b(x), h_b(x) \in \mathcal{M}
\right\}
\end{equation}

That's why for our comfort we will define the operators $L$ and $\Phi$ acting on $\mathcal{M}$ as
\begin{align}\label{eq_def_L_and_Phi_normalized}
L h:= \frac{\bar{L} [x^2 h]}{x^2}  \;\;\;\;\; \;\;\;\;\;\;\;
\Phi h:= \frac{\bar{\Phi} [x^2 h]}{x^2}
\end{align}

These operators correspond to the solution of homological equation and to the shift operator but are somewhat normalized.

The linearization theorems we prove will be applicable only in the vicinity of the base, i.e. in $\mathbb{T}^d \times [0, \varepsilon]$. The conditions on the small constant $\varepsilon$ will be formulated later.
% НЕ ЗАБЫТЬ
 Contraction mapping theorem  will be applied to the operator $L \Phi$ acting in the complete metric space $\mathcal{M}_{\varepsilon}$ of functions from $\mathcal{M}$ restricted to the small neighborhood of a torus $\mathbb{T}^d \times [0,\varepsilon]$. A norm on this space is simply a continuous one, for $h(b,x) \in \mathcal{M}_{\varepsilon}$ it is defined by
\begin{equation}
||h||_{C, \varepsilon}=\sup_{(b,x) \in \mathbb{T}^d \times [0,\varepsilon]} \left|h_b(x)\right|.
\end{equation}
%We will write $||h||_C$ for $||h||_{C, 1}$ to shorten the notations.
To use contraction mapping principle we define a space
\begin{equation}\label{eq_definition_space_and_metric}
\mathcal{N}:=\left\{h \in \mathcal{M}_{\varepsilon}, ||h||_{C, \varepsilon} \leq A\right\}, \rho(h_1, h_2):=||h_1-h_2||_{C,\varepsilon}.
\end{equation}
with a continuous norm on it.

The constant $A$ will be chosen later. Now we pass to the definition of the set $N$.
% but first we will need some notations.

\subsection{H\"{o}lder property and a closed subspace $N$}
To prove Theorem \ref{thm:new:2} we shall work with the three norms: continuous one $|| \cdot ||_{C, \varepsilon}$ was already defined, now we will define the Lipschitz norm $\Lip_{x, \varepsilon}$ as well as the H\"{o}lder one $||\cdot||_{[\alpha], \varepsilon}$. The index $\varepsilon$ indicates that these norms are considered for the subspaces of functions in $\mathcal{M}^2_{\varepsilon}$; but it will be omitted in the case it is matter-of-course.

\begin{defn}
For a function $h \in \mathcal{M}$ define its \emph{H\"older norm }$||h||_{[\alpha]}$ as
\begin{equation}
||h||_{[\alpha]}:=\sup_{b_1,b_2 \in \mathbb{T}^d, x \in [0,1]} \frac{\left|h_{b_1}(x)-h_{b_2}(x)\right|}{||b_1-b_2||^{\alpha}}
\end{equation}

H\"older norm of a function is sometimes called its H\"older constant.

The subspace of functions $h  \in \mathcal{M}$ for which this norm is finite, will be called the space of \emph{H\"{o}lder functions with exponent} $\alpha$ and denoted by $\mathcal{H}^{\alpha}$. In much the same way, the space $\mathcal{H}^{\alpha}_{\varepsilon}$ is a subset of functions $h$ in $\mathcal{M}_{\varepsilon}$ such that $||h||_{[\alpha], \varepsilon}<\infty$ where

\begin{equation}
||h||_{[\alpha],\varepsilon}:=\sup_{b_1,b_2 \in \mathbb{T}^d, x \in [0,\varepsilon]} \frac{\left|h_{b_1}(x)-h_{b_2}(x)\right|}{||b_1-b_2||^{\alpha}}
\end{equation}
\end{defn}

\begin{defn}
For a function $h \in \mathcal{M}$ define its fiberwise \emph{Lipschitz norm} $\Lip_x h$ as

\begin{equation}
\Lip_x h := \sup_{b \in \mathbb{T}^d, x,y \in [0,1]} \frac{\left| h_b(x)-h_b(y)\right|}{|x-y|}
\end{equation}
Analogously, for $h \in \mathcal{M}_{\varepsilon}$
\begin{equation}
\Lip_{x, \varepsilon} h:=\sup_{b \in \mathbb{T}^d, x,y \in [0,\varepsilon]} \frac{\left| h_b(x)-h_b(y)\right|}{|x-y|}
\end{equation}
\end{defn}

Once these definitions given, we can say what will be the closed subspace $N$ of the functional space $\mathcal{N}$ (see \eqref{eq_definition_space_and_metric}) for the contraction mapping principle. We will show that there exist constants $\varepsilon>0$ as well as $A_{C}, A_{\Lip}$ and $A_{\alpha}$ such that the space

\begin{equation}\label{eq_definition_closed_subspace_N}
N=\left\{
h \in \mathcal{M}_{\varepsilon}, h \in \Hae \; |  \;||h||_{C,\varepsilon} \leq A_{C}, \Lip_{x,\varepsilon} h \leq A_{\Lip}, ||h||_{[\alpha], \varepsilon} \leq A_{\alpha}
\right\}
\end{equation}
is closed in $(\mathcal{N}, \rho)$ and preserved under $L \Phi$. Note that here we need $A \geq A_C$. In the proof, we will first choose constants $A_{C}, A_{\Lip}$ and $A_{\alpha}$, $A$ can be chosen later as $A:=A_C$.

\subsection{Three main lemmas and the proof of Theorem \ref{thm:new:2}}
To prove Theorem \ref{thm:new:2}, one needs simply to show that all the conditions of contraction mapping principle hold for $\mathcal{N}$, $\rho$ and $N$ defined correspondingly in \eqref{eq_definition_space_and_metric} and \eqref{eq_definition_closed_subspace_N}. Here we state three main lemmas that will give the result of Theorem \ref{thm:new:2}.

 Lemma \ref{lem_homological} deals with homological equation and provides an explicit solution of \eqref{eq_homological_equation_short} as a formal series. It also states that this series converges exponentially and gives a continuous function on $M$. Moreover, for $\alpha$ chosen accordingly to \eqref{eq:new:alpha}, the operator $L$ in the space $\mathcal{M}$ preserves the subspace $\mathcal{H}^{\alpha}$ of H\"older functions with this particular exponent. This is a crucial point that gives us the main claim -- H\"older property of a conjugacy.

The two lemmas that are left enable us to apply contraction mapping principle. Lemma \ref{lem_shift} deals with composition $L \Phi$: it states that one can choose a closed subspace $N \subset \mathcal{N}$ of the form \eqref{eq_definition_closed_subspace_N} such that is mapped into itself under the composition $L \Phi$.
Lemma \ref{lem_contraction} proves that $L \Phi$ is indeed a contraction on the space $\Me$ in continuous norm.

Let us state precisely these lemmas.

\begin{lem}\label{lem_homological}[Solution of a homological equation]
Consider a skew product \eqref{eq:new:def_skew}. Let us define a sequence of functions on $\mathbb{T}^d$ as
\begin{equation}\label{eq_definition_pi}
\Pi_0(b):=1, \Pi_n(b):=\lambda_b \lambda_{Ab}\ldots \lambda_{A^{n-1}b}, n=1, 2, \ldots
\end{equation}
%Suppose $h \in M$ is a solution of normalized homological equation in $\M$ with $Q$ on the right-hand side: $L Q =h$.
Let $\alpha$ be given by \eqref{eq:new:alpha}, and set
\begin{equation}\label{eq_theta}
\theta= \theta(\alpha):=\mu^{\alpha} q <1
\end{equation}
Suppose that conditions \eqref{eq:lambda_bound} and \eqref{eq_theta} hold, and let $Q \in \mathcal{H}^{\alpha}$.

Then the following holds:

$1.$ There exist a solution $h_b(x)$ of the homological equation \eqref{eq_homological_equation_Q}; it can be represented as a formal series
\begin{equation}\label{eq_solution_of_homological_equation_formal_series}
h_b(x)=-\sum_{k=0}^{\infty}\frac{\Pi_k(b) Q \circ F_0^k(b,x)}{\lambda_{A^{k} b}}
\end{equation}

$2.$ The series \eqref{eq_solution_of_homological_equation_formal_series} converges uniformly on $M$ and its sum is continuous in $b$ and as smooth in $x$ as $Q$.

$3.$ The solution $h$ satisfies H\"older condition with the exponent equal to $\alpha$ : $h \in \mathcal{H}^{\alpha}$.

$4.$ The operator $L: Q \mapsto h$ is bounded in $C$--norm on the space $\M$.
\end{lem}

\begin{lem}\label{lem_shift}[A closed subspace maps inside itself]
For a skew product of the form \eqref{eq:new:def_skew} there exist constants $\varepsilon, A_C, A_{\Lip}, A_{\alpha}>0$ such that the operator $L \Phi$ acting in the space $\M$ maps the closed space $N$ defined by \eqref{eq_definition_closed_subspace_N} into itself.
\end{lem}

\begin{lem}\label{lem_contraction}[Contraction property]
There exists a constant $A>0$ such that for any sufficiently small $\varepsilon>0$ the operator $L \Phi$ acting in the space $\mathcal{N}$ (which depends on $A$ and on $\varepsilon$, see \eqref{eq_definition_space_and_metric}) is contracting in the continuous norm.
\end{lem}

\emph{Proof of Theorem \ref{thm:new:2}}.

Now the proof follows: first, we take $\varepsilon$ defined by Lemma \ref{lem_shift} and fix all constants $A_C, A_{\Lip}, A_{\alpha}$ provided by the same lemma. Then we diminish $\varepsilon$ for Lemma \ref{lem_contraction} to hold. Then the set $\mathcal{N}$  corresponding to such an $\varepsilon$ and a constant $A=A_C$ has a $C$-norm on it defining a complete metric space. Operator $L \Phi$ acts in this space and, by Lemma \ref{lem_contraction}, is a contracting map. Note that the set $N$ defined in \eqref{eq_definition_closed_subspace_N} is a closed subspace of $\mathcal{N}$ since $A=A_C$. This subspace $N \subset{\Hae}$ with a fixed H\"older constant $\alpha$ is preserved by $L \Phi$. Then, by contraction mapping principle, $L \Phi$ has a fixed point $h \in N$ (and hence in $\Hae$) which gives H\"older conjugacy of the initial skew product with its linearization. Strictly speaking, the H\"older property is proven not for the conjugacy but for its quadratic part divided by $x^2$ but the H\"older property of the conjugacy follows just because $x$ is bounded.
\hfill $\Box$

\section{Proof of Lemma \ref{lem_homological}: homological equation solution}\label{section_homological}
From the form \eqref{eq_homological_equation_Q} of the homological equation we deduce that ${h}(b,x)$ can be represented as
\begin{equation}\label{eq_homological_representation}
h=-\lambda^{-1} Q + \lambda h \circ F_0
\end{equation}
Let us take the right composition of this equation with the normalized map $F_0$ given by \eqref{eq_linearization_form}. And then let us apply the operator of multiplication by $\lambda$ to this equation. The equality \eqref{eq_homological_representation} implies
\begin{equation}\label{eq_homological_representation_second_term}
\lambda \left(h \circ F_0\right) = -\lambda\left(\lambda^{-1} \circ A\right) Q \circ F_0 + \lambda \left(\lambda \circ A \right) h \circ F_0^2
\end{equation}

Note that the left side of \eqref{eq_homological_representation_second_term} is equal to one of the terms in the right hand side of \eqref{eq_homological_representation}. We continue such a process of taking right composition with $F_0$ and multiplying by $\lambda$. Thus we obtain the infinite sequence of equations that can be all summed up. Let us sum the first $N+1$ of them, then we will have

\begin{equation}\label{eq_for_h}
h_b=\Pi_{N+1}(b) h_b \circ F_0^{N+1} -  \sum_{k=0}^{N} \frac{\Pi_k(b) Q \circ F_0^k}{\lambda_{A^k b}}
\end{equation}

Let us now pass to the limit when $N \rightarrow \infty$: since $h \in \mathcal{N}$, $||h||_C \leq A$ and $\lambda$ is bounded by some $q<1$, we have that the first term on the right-hand side of \eqref{eq_for_h} is bounded by $A q^{N+1}$ and hence tends to $0$. Thus we obtain formula \eqref{eq_solution_of_homological_equation_formal_series} for $h(b,x)$.

Since $F$ is a diffeomorphism, then $\forall b \in \Td$ we have: $\lambda_b \neq 0$. Then, since $\lambda_b$ is a continuous function on a compact manifold $\Td$, there exists a lower bound $D>0$ such that
 \begin{equation}
\lambda_b \geq D>0 \;\;\;\;\; \forall b \in \Td.
 \end{equation}
Then, since obviously
\begin{equation}\label{eq_obvious_2}
|\Pi_k(b)|\leq q^k ,
\end{equation}
the series \eqref{eq_solution_of_homological_equation_formal_series} is bounded by a converging number series
\begin{equation}\label{eq_5}
\sum_{k=0}^{\infty} \frac{q^k}{D} ||Q||_C= \frac{||Q||_C}{D (1-q)}
\end{equation}
So by the Weierstrass majorant theorem, its sum is a continuous function on $M$, and the normalized homological operator $L$ is bounded in continuous norm. Namely,

\begin{equation}\label{eq_marie}
||L||_C \leq \frac{1}{D (1-q)}
\end{equation}

Solution $h_b(x)$ is as smooth in $x$ as $Q$ is: that can be verified by differentiation of the series \eqref{eq_solution_of_homological_equation_formal_series} and repetitive applying of Weierstrass majorant theorem. Convergence of the series for the derivative of the solution of homological equation $h$ will be even faster than the convergence of the series for the function itself: indeed, the coefficients in the series \eqref{eq_solution_of_homological_equation_formal_series} will be multiplied by the factors $\Pi_k(b)$ which are rapidly decreasing.

So assertions $1, 2$ and $4$ of Lemma \ref{lem_homological} are proven. What is left to prove is that H\"older property with exponent $\alpha$ is preserved by operator $L$. We will need the following

\begin{prop}
In the setting of Theorem \ref{thm:new:1}, let the H\"older property \eqref{eq:new:Holder_f} for $f_b$ and some $k$ hold. Then, for $\lambda_b$ and $Q_b(x)$ given by $f_b(x) = \lambda_b x+x^2 Q_b(x)$ we have H\"older properties: for $\lambda_b$ and for the same $k$ as in \eqref{eq:new:Holder_f}; for $Q_b(x)$ and for $k-2$.
\end{prop}

\begin{proof}
The property for $\lambda_b$ is obvious since $\lambda_b = \frac{\partial f_b(x)}{\partial x} \left| \right. _{x=0}$. The property for $Q_b(x)$ follows from an analogue of Hadamard's lemma:
for $\varphi \in C^2_{[0,1]}, \varphi(0)=\varphi'(0)=0$ and $\psi=\frac{\varphi}{x^2}$ we have
\begin{equation}\label{estimate}
|| \psi ||_{C} \leq || \varphi ||_{C^2}
\end{equation}
 This follows from the well-known formula $\psi(x)=\int_0^x (x-t) \varphi''((t)) dt$. The coordinate change $t=xs, s \in [0,1]$ implies $\psi(x)=x^2 \int_0^1 (1-s) f''(xs) ds$. From this \eqref{estimate} follows. The needed corollary that $Q_b$ is H\"older continuous assuming $f_b$ is H\"older as an element of $C^2$, follows.
\end{proof}

To prove assumption $3$ of Lemma \ref{lem_homological} let us denote by $C_Q:=||Q||_{[\alpha]}$ and $C_{\lambda}:=||\lambda||_{[\alpha]}$ the H\"older constants for functions $Q$ and $\lambda$ respectively. We need to find such $C>0$ that for all $b_1, b_2 \in \Td$:

\begin{equation}\label{eq_definition_of_Holder_constant_for_h}
|h_{b_1}(x)-h_{b_2}(x)| \leq C ||b_1-b_2||^{\alpha}
\end{equation}

Note that even though H\"older exponents for $Q$ and $\lambda$ can be close to $1$, the H\"older exponent for the solution $h$ of normalized homological equation will be close to zero.

For each $k \in \Z_{+}$ denote
\begin{equation}\label{eq_definition_P_n}
P_k(b):=\frac{\Pi_k(b)}{\lambda_{A^k b}}
\end{equation}
Then, obviously,
\begin{equation}\label{eq_obvious}
|P_k(b)| \leq q^k D.
\end{equation}
Let $Q_{k}(b,x):=Q \circ F_0^{k} (b,x)$. Then the solution $h$ can be written in the form
\begin{equation}
h_b(x)=-\sum_{k=0}^{\infty} P_k(b) Q_k(b,x)
\end{equation}
Take $b_1, b_2 \in \Td$ and denote $Q_{k,j}:=Q \circ F_0^{k} (b_j, x), j=1,2$.
Then
\begin{equation}
\left|h_{b_1}(x)-h_{b_2}(x)\right| = \sum_{k=0}^{\infty} \left[
(P_k(b_1)-P_k(b_2))Q_{k,1}+P_k(b_2) (Q_{k,1}-Q_{k,2})
\right]
\end{equation}
So we have the estimate
\begin{equation}\label{eq_main_estimate_Holder}
\left|h_{b_1}(x)-h_{b_2}(x)\right|  \leq \sum_{k=0}^{\infty} \theta_{1,k}(b_1,b_2)+\theta_{2,k}(b_1,b_2)
\end{equation}
where
\begin{equation}\label{def_theta}
\theta_{1,k}(b_1,b_2)=\left|P_k(b_1)-P_k(b_2)\right| ||Q||_C,  \,\,\theta_{2,k}(b_1,b_2)=|P_k(b_2)| |Q_{k,1}-Q_{k,2}|
\end{equation}

Let us formulate some propositions that we will need, and postpone their proofs to the Appendix, Section \ref{section_appendix}.
\begin{prop}\label{proposition_Pi_n_is_Holder}
Function $\Pi_n(b)$ defined as a product of $\lambda_b$ in the first $n$ points of the orbit of a linear diffeomorphism $A$, see \eqref{eq_definition_pi}, is H\"older continuous with the exponent $\alpha$, see \eqref{eq:new:alpha} ,  and
\begin{equation}||\Pi_n||_{[\alpha]} \leq
%q^{n-1} C_{\lambda} \frac{\mu^{n \alpha}}{\mu^{\alpha}-1}
\frac{C_{\lambda} \theta^n}{(\mu^{\alpha}-1)q}
%  ||b_1-b_2||^{\alpha}
\end{equation}
where $C_{\lambda}$ is the H\"older constant for $\lambda$, $\theta$ is defined in \eqref{eq_theta}, and $\mu$ is the largest magnitude of eigenvalues of $A$.
\end{prop}

\begin{prop}\label{proposition_P_n_is_Holder}
Function $P_n(b)$ defined by \eqref{eq_definition_P_n} is H\"older with the exponent $\alpha$, and
\begin{equation}||P_n||_{[\alpha]} \leq D^2 C_{\lambda} B \theta^n,
\end{equation}where $B$ depends only on the initial map, the precise formula for $B$ is given below, see \eqref{eq_B}.
\end{prop}

Now, using Proposition \ref{proposition_P_n_is_Holder} we can prove that

\begin{equation}\label{eq_theta_one}
\theta_{1,k} (b_1, b_2) \leq ||Q||_C D^2  C_{\lambda} B \theta^k \go
\end{equation}

The estimate of $\theta_{2,k}$ is somewhat lengthier.

\begin{prop}\label{prop_estimate_theta_two}
Function $\theta_{2,k}(b_1,b_2)$ defined in \eqref{def_theta} is H\"older with the exponent $\alpha$, and
\begin{equation}\label{eq_49}
|| \theta_{2,k}||_{[\alpha]} \leq \theta^k D
\left(  C_Q + q^{k-1} \Lip_x Q \frac{C_{\lambda}}{\mu^{\alpha}-1}\right)
\end{equation}
\end{prop}
The proof of this proposition is using only the triangle inequality and we postpone it till the appendix.

Inserting estimates on $\theta_{1,k}$ and $\theta_{2,k}$ from \eqref{eq_theta_one} and \eqref{eq_49} into the inequality \eqref{eq_main_estimate_Holder}, we can finally use our special choice of $\alpha$. It is in this place where we crucially use the fact that $\theta<1$ to establish the convergence of estimating series in the right-hand side of \eqref{eq_main_estimate_Holder}. By simple computation of the sum of a geometric progression, we obtain that $h$ is H\"older, and \eqref{eq_definition_of_Holder_constant_for_h} holds for some $C_h$. The expicit form of $C_h$ is not important for the proof of this lemma, but it will be used in the proof of Lemma \ref{lem_shift}. That's why we write it out explicitely:
\begin{equation}\label{eq_Holder_estimate_for_homological_operator}
C_h=||Q||_C  L_C+C_Q L_{[\alpha]}+\Lip_x Q L_{\Lip}.
\end{equation}
where
\begin{equation}\label{eq_10}
L_C=\frac{D^2 C_{\lambda} B}{1-\theta}, L_{[\alpha]}=\frac{D}{1-\theta}, L_{\Lip}=\frac{C_{\lambda}}{\left(\mu^{\alpha}-1\right) q}\frac{1}{1-\theta q}.
\end{equation}
This completes the proof of Lemma \ref{lem_homological}.
\hfill $\Box$
\section{Proof of Lemma \ref{lem_shift}: the shift operator}\label{section_shift}
Take some $h \in N$ and let us estimate continuous, Lipschitz and H\"older norms of its image under the composition of operators $L$ and $\Phi$.

The plan of the proof is the following: we will first show that there exist constants $\varepsilon_C>0$ and $A=A_C>0$ such that the space $\mathcal{N}$ defined by \eqref{eq_definition_space_and_metric} is mapped by $L \Phi$ to itself. So the operator $L \Phi$ doesn't increase too much the continuous norm if we consider it on an appropriate space.

In the following step, we will diminish even more the $\varepsilon$-neighborhood of the base in which the functions are defined, and search for $\varepsilon_{\Lip}<\varepsilon_C$ and we will also search for a good bound $A_{\Lip}$ in \eqref{eq_definition_space_and_metric}. We will find such $\varepsilon_{\Lip}$ and $A_{\Lip}$ that $L \Phi$ won't increase the Lipschitz norm of the function $h$ with conditions $||h||_C \leq A_C, ||h||_{\Lip}\leq A_{\Lip}$ in the vicinity of the base.

And, in the final step,  we will find $\varepsilon_{\alpha}<\varepsilon_{\Lip}$ and $A_{\alpha}$ such that the space $N$ defined by \eqref{eq_definition_closed_subspace_N} is preserved by $L \Phi$.

From the definition \eqref{eq_shift_operator_def} of the shift operator $\bar{\Phi}$ we have

\begin{equation}
\bar{\Phi}\bar{h} (b,x)=R_b \left(x+\bar{h}_b(x)\right)= (x+\bar{h}_b(x))^2 Q(b, x+x^2 h_b(x))=x^2 (1+x h_b (x))^2 Q(b, x+x^2 h_b(x))
\end{equation} hence

\begin{equation}\label{eq_shift_operator_normalized_form}
\Phi h=(1+x h)^2 Q(b,x+x^2 h)
\end{equation}
Using the definition \eqref{eq_definition_closed_subspace_N} of the subspace $N$ as well as the estimate \eqref{eq_5} and the expression \eqref{eq_shift_operator_normalized_form}, for any $h \in N$ we have

\begin{equation}\label{eq_6}
||L \Phi h||_{C, \varepsilon} \leq \frac{1}{D(1-q)} ||\Phi h||_{C, \varepsilon} \leq \frac{||Q||_C}{D (1-q)} \left(1+\varepsilon A_C\right)^2
\end{equation}

Hence let us first fix any
\begin{equation}\label{eq_7}
A_C > \frac{||Q||_C}{D(1-q)}
\end{equation}
and then choose $\varepsilon=\varepsilon_C$ such that

\begin{equation}
\frac{||Q||_C}{D(1-q)} (1+\varepsilon A_C)^2<A_C
\end{equation}

Note that in the definition of the space $\mathcal{N}$ the constant $A$ bounding the norm should be greater than $A_C$ defined by \eqref{eq_7}.

For the Lipschitz norm bound, we will need the proposition concerning only the homological operator: it preserves the space of smooth on fiber functions. Since we will deal with derivatives of functions along the fiber let us agree on notations: let us denote the $l$-th derivative of a function $h(b,x)$ with respect to fiber coordinate $x$ as $h^{(l)}, l\in \mathbb{N}$.

\begin{prop}\label{prop_L_Lipschitz}
The operator $L$ is bounded in the Lipschitz norm: there exists a constant $\Lip_x L$ such that for any $h \in \mathcal{M}$ the following holds:
 $$\Lip_x \left( Lh \right) \leq \Lip_x L \cdot  \Lip_x h.$$
Moreover, if for any $b \in \Td$, the function $h(b,\cdot) \in C^l$, then $Lh$ has the same smoothness as well and
\begin{equation}\label{eq_derivative}
\left|\left|(Lh)^{(l)}\right|\right|_C \leq C_k(L) \left|\left|h^{(l)}\right|\right|_C
\end{equation}
\end{prop}

The proof of this proposition is an easy consequence of the explicit form \eqref{eq_solution_of_homological_equation_formal_series} for the solution of the normalized homological equation, and we give it in the Appendix, Section \ref{section_appendix}.

Now let us pass to the Lipschitz norm $ \Lip_{x,\varepsilon} [L \Phi h] \leq \Lip_x L  \times \Lip_{x, \varepsilon} \Phi h$. By using the simple arguments one can prove the following

\begin{prop}\label{prop_lip_calc}
There exist polynomials $T_3(\varepsilon)$ and $T_4^0(\varepsilon)$ of degrees respectively $3$ and $4$ such that $T_4^0(0)=0$ and for any $h \in N$ holds
\begin{equation}
\Lip_{x,\varepsilon} [\Phi h] \leq T_3(\varepsilon)+T_4^0(\varepsilon) A_{\Lip}
\end{equation}
\end{prop}
We postpone the proof to the Appendix.

From here we see that there exists a constant $A_{\Lip}$ such that for $\varepsilon$ small enough, say $\varepsilon<\varepsilon_{\Lip}$, Lipschitz constant of the image of any function $h \in N$ is bounded by $A_{\Lip}$:

 $$\Lip_{x,\varepsilon} [L \Phi h] \leq A_{\Lip}.$$
We can assume that $\varepsilon_{\Lip} < \varepsilon_C$.

What is left is to estimate $ ||L \Phi h||_{[\alpha], \varepsilon}$: for this, we will need the bounds on how operators $L$ and $\Phi$ behave on the space of $\alpha$-H\"older functions separately.

For the shift operator in Appendix, Section \ref{section_appendix} we will prove

\begin{prop}\label{prop_6}
If $h \in \mathcal{H^{\alpha}}_{\varepsilon}$ then $\Phi h \in \mathcal{H^{\alpha}}_{\varepsilon}$ as well. And, moreover, for $h \in N$, there exist polynomials $\tilde{T}_2(\varepsilon)$ and $\tilde{T}_4^0(\varepsilon), \tilde{T}_4^0(\varepsilon)(0)=0$ of degrees $2$ and $4$ correspondingly such that
\begin{equation}\label{eq_12}
\left|\left|\
\Phi h
\right|
\right|_{[\alpha], \varepsilon} \leq \tilde{T}_4^0(\varepsilon) A_{\alpha}+ \tilde{T}_2(\varepsilon)
\end{equation}
\end{prop}

While proving Lemma \ref{lem_homological}, we have deduced the bound \eqref{eq_Holder_estimate_for_homological_operator} on H\"older norm of normalized homological operator with $L_C, L_{[\alpha]}, L_{\Lip}$ being some fixed constants defined by \eqref{eq_10}:

\begin{equation}\label{eq_11}
||Lh||_{[\alpha], \varepsilon} \leq L_C ||h||_{C,\varepsilon} + L_{[\alpha]} ||h||_{[\alpha], \varepsilon} + L_{\Lip} \Lip_{x, \varepsilon} h
\end{equation}

Let us now combine \eqref{eq_12} and \eqref{eq_11} for $h:=\Phi h$ to get the bound for $ ||L \Phi h||_{[\alpha], \varepsilon}$. Here we will be using Propositions \ref{prop_lip_calc} and \ref{prop_6} as well as inequality \eqref{eq_6}
to get the bounds on different norms of $\Phi h$ in the space $\mathcal{M}_{\varepsilon}$.

\begin{multline}
||L \Phi h||_{[\alpha], \varepsilon} \leq L_C \left|\left|\Phi h\right|\right|_{C, \varepsilon}+L_{[\alpha]}\left|\left|\Phi h\right|\right|_{[\alpha],\varepsilon}+L_{\Lip} \Lip_{x, \varepsilon} \Phi h \leq\\ \leq L_C ||Q||_C (1+\varepsilon A_C)^2 + L_{[\alpha]}\left(\tilde{T}_2^0(\varepsilon)A_{\alpha}+\tilde{T}_2(\varepsilon)\right)+
L_{\Lip} \left(T_2(\varepsilon)+T_4^0(\varepsilon)A_{\Lip} \right)
\end{multline}

So we see that there exist polynomials $Q_2^0(\varepsilon), Q_4(\varepsilon)$ such that $\deg Q_2^0=2, Q_2^0(0)=0, \deg Q_4(\varepsilon)=4$ and

\begin{equation}\label{eq_9}
||L \Phi h||_{[\alpha], \varepsilon} \leq A_{\alpha} Q_2^0(\varepsilon)+Q_4(\varepsilon)
\end{equation}

So for $\varepsilon$ small enough, $\varepsilon<\varepsilon_{[\alpha]}$, and for some $A_{\alpha}>0$ the right-hand side of inequality \eqref{eq_9} can be made less than $A_{\alpha}$. We can take $\varepsilon_{[\alpha]}<\varepsilon_{\Lip}$. By taking $\varepsilon=\varepsilon_{[\alpha]}$ we obtain the desired preservation of $N$ by operator $L \Phi$. This space is obviously closed in $\mathcal{N}$.

\hfill $\Box$
\section{Proof of Lemma \ref{lem_contraction}: contraction property}\label{section_contraction}
Since operator $L$ is linear and uniformly bounded by \eqref{eq_5} in the continuous norm, the only thing to prove is that normalized shift operator $\Phi$ is strongly contracting in this norm, i.e. for any $\varepsilon$ small enough there exists some constant $\nu=\nu(\varepsilon) \in (0,1)$ such that for any $h,g \in \mathcal{N}$

\begin{equation}
||\Phi h - \Phi g||_{C, \varepsilon} \leq \nu ||h-g||_{C, \varepsilon}
\end{equation}

\emph{Proof.}

Suppose $h, g \in \mathcal{M}$ and define $\bar{h}, \bar{g} \in \mathcal{M}^2$ by $\bar{h}_b(x)=x^2 h(b,x),\bar{g}_b(x)=x^2 g_b(x)$. Also denote $Q_h=Q\left(b, x+\bar{h}_b(x)\right)$.
\begin{multline}
\left|\left|
\Phi h - \Phi g
\right|\right|_{C, \varepsilon}=
\left|\left|
(1+xh_b(x))^2Q_h-(1+xg_b(x))^2Q_g)
\right|\right| \leq\\ \leq
\left|\left|
Q_h-Q_g
\right|\right|_{C, \varepsilon}+\left|\left|
2xh_b(x)Q_h-2xg_b(x)Q_g
\right|\right|_{C, \varepsilon}+
\left|\left|
x^2h^2_b(x) Q_h - x^2 g^2_b(x) Q_g
\right|\right|_{C, \varepsilon} \leq\\
\leq \Lip_x Q ||\bar{h}-\bar{g}||_{C, \varepsilon} + 2 \varepsilon ||h-g||_{C, \varepsilon} ||Q||_C + 2 \varepsilon A ||Q_h-Q_g||_{C, \varepsilon}
 +\varepsilon^2 \left|\left|(h^2-g^2) Q_h+g^2(Q_h-Q_g)\right|\right|_{C, \varepsilon} \leq \\
\left|\left|
h-g
\right|\right|_{C, \varepsilon}
\left(
\varepsilon^2 \Lip_x Q + 2 \varepsilon ||Q||_C + 2 \varepsilon^2 \Lip_x Q A
\right)
+ \varepsilon^2 \left(2 A ||h-g||_{C, \varepsilon} ||Q||_C + A^2 \Lip_x Q \varepsilon^2 ||h-g||_{C, \varepsilon}\right)=\\=
\left|\left|
h-g
\right|\right|_{C, \varepsilon} o(\varepsilon).
\end{multline}
Hence operator $\Phi$ is strongly contracting. And since from \eqref{eq_marie} for any function $h \in \mathcal{N}$ the norm $||Lh||_{C,\varepsilon} \leq \frac{D}{1-q} ||h||_{C, \varepsilon}$, applying this to $\Phi h$ with the fact of the strong contraction property for $\Phi$ we get the strong contraction property for $L \Phi$.
% and hence operator $L \Phi$ is contracting in the space $\mathcal{N}$ for $\varepsilon$ small enough.
\hfill $\Box$
\section{Appendix: Proof of Theorem \ref{thm:new:3} and other calculations}\label{section_appendix}
In the appendix we will prove the technical propositions stated above.
\subsection{H\"older properties of some auxiliary functions}
First let us prove

\textbf{Proposition 2.}
\emph{Function }$\Pi_n(b)$ \emph{defined as a product of }$\lb$ i\emph{n the first} $n$ \emph{points of the orbit of a linear diffeomorphism }$A$, \emph{see }\eqref{eq_definition_pi}, \emph{is} \emph{H\"older with the exponent} $\alpha$ \emph{and }
\begin{equation}||\Pi_n||_{[\alpha]} \leq
%q^{n-1} C_{\lambda} \frac{\mu^{n \alpha}}{\mu^{\alpha}-1}
\frac{C_{\lambda} \theta^n}{(\mu^{\alpha}-1)q}
%  ||b_1-b_2||^{\alpha}
\end{equation}
\emph{where} $C_{\lambda}$ \emph{is H\"older constant for }$\lambda$, $\theta$ \emph{is defined in} \eqref{eq_theta}. \emph{Here and below} $\alpha$ \emph{is given by} \eqref{eq:new:alpha} \emph{and }$\mu$ is the largest magnitude of eigenvalues of $A$.

Proof of Proposition 2:
\begin{multline}
\left|
\Pi_n(b_1)-\Pi_n(b_2)
\right| =
\left|
\prod_{k=0}^{n-1} \lambda_{A^{k} b_1}-\prod_{k=0}^{n-1} \lambda_{A^k b_2}
\right|=
\left|
\lambda_{b_1}-\lambda_{b_2}
\right|
\times
\left|
\prod_{k=1}^{n-1} \lambda_{A^k b_1}
\right|+
\left|\lambda_{b_2}\right|
\left|
\Pi_{n-1}(A b_1) - \Pi_{n-1} (A b_2)
\right| \leq \ldots \\ \leq
q^{n-1} C_{\lambda} \sum_{k=0}^{n-1} \left|\left|A^k b_1-A^k b_2 \right|\right|^{\alpha}
\leq q^{n-1} C_{\lambda} \frac{\mu^{n \alpha}-1}{\mu^{\alpha}-1} \go  \leq \frac{C_{\lambda} \theta^n}{(\mu^{\alpha}-1)q} \go
\end{multline}
\hfill $\Box$

\textbf{Proposition 3.}
\emph{Function }$P_n(b)$ defined by $P_n(b):=\frac{\Pi_n(b)}{\lambda_{A^n b}}$ \emph{is H\"older with exponent }$\alpha$\emph{ and }
\begin{equation}||P_n||_{[\alpha]} \leq D^2 C_{\lambda} B \theta^n
%здесь оценочку
\end{equation}\emph{where} $B$ \emph{depends only on the initial map, the precise formula for }$B$ \emph{is given below, see }\eqref{eq_B}.

\begin{proof}
%см страницу 8 моего старого текста
\begin{multline}
\left|
P_n(b_1)-P_n(b_2)
\right| =
\left|
\frac{\Pi_n(b_1) \lambda_{A^n b_2}-\Pi_n(b_2) \lambda_{A^n b_1}}{\lambda_{A^n b_1} \lambda_{A^n b_2}}
\right|  \leq
D^2
\left|
\lambda_{A^n b_2}\prod_{k=0}^{n-1} \lambda_{A^{k} b_1} \right.-
\left. \lambda_{A^n b_1}\prod_{k=0}^{n-1} \lambda_{A^{k} b_2}
\right|=\\=
D^2
\left|
\left(\lambda_{A^n b_2}-\lambda_{A^n b_1}\right) \Pi_n(b_1)+\Pi_{n+1}(b_1)
\right.
-\left.\left(\lambda_{A^n b_1}
-\lambda_{A^n b_2}\right) \Pi_n(b_2)-\Pi_{n+1}(b_2))
\right| \leq \\
\leq
|\Pi_{n+1} (b_1)-\Pi_{n+1}(b_2)|+\left|\lambda_{A^n b_1}-\lambda_{A^n b_2}| |\Pi_n(b_1)-\Pi_n(b_2)
\right| \leq \\
\leq D^2 \left[
\frac{C_{\lambda} \theta^{n+1}}{(\mu^{\alpha}-1)q}+2 q^n C_{\lambda} \mu^{n \alpha}
\right] \go \leq D^2 C_{\lambda} B \theta^n \go
\end{multline}
where $B$ doesn't depend on anything but initial skew product:
\begin{equation}\label{eq_B}
B(\theta,\mu,\alpha,q)=\frac{\theta}{(\mu^{\alpha}-1)q}+2
\end{equation}
\end{proof}

\textbf{Proposition 4.}
\emph{Function} $\theta_{2,k}(b_1,b_2)$ \emph{defined as}
$\theta_{2,k}(b_1,b_2)=|P_k(b_2)| |Q_{k,1}-Q_{k,2}|$ \emph{is H\"older with }$\alpha$ \emph{as exponent and}
\begin{equation}
||\theta_{2,k}||_{[\alpha]} \leq \theta^k D
\left(  C_Q + q^{k-1} \Lip_x Q \frac{C_{\lambda}}{\mu^{\alpha}-1}\right)
\end{equation}
\emph{Here }$Q_{k,1}=Q \circ F_0^{k}(b_1,x)$ and $Q_{k,2}=Q \circ F_0^{k}(b_2,x)$, \emph{and the definition of} $P_k(b)$ \emph{was reminded in Proposition} $2$ \emph{above}.
\begin{proof}
We use the results of Proposition \ref{proposition_Pi_n_is_Holder} in the following chain of inequalities.
\begin{multline}\label{eq_theta_two}
\theta_{2,k} \leq q^k D \left| Q_{A^k b_1}( \Pi_k(b_1)x)-Q_{A^k b_2}( \Pi_k(b_2)x)\right|
\leq  q^k D
\left| Q_{A^k b_1}(\Pi_k(b_1)x)-Q_{A^k b_2}( \Pi_k(b_1)x)\right|+\\+q^k D \left|
Q_{A^k b_2}(\Pi_k(b_2)x)-Q_{A^k b_2}(\Pi_k(b_1)x)
\right|
\leq q^k D C_Q \mu^{k \alpha} \go +q^k D \Lip_x Q ||\Pi_k||_{\mathcal{H}^{\alpha}}\go \leq\\
\theta^k D
\left(  C_Q + q^{k-1} \Lip_x Q \frac{C_{\lambda}}{\mu^{\alpha}-1}\right) \go
\end{multline}
\end{proof}

\textbf{Proposition 5.}
\emph{Operator} $L$ \emph{is bounded in the Lipschitz norm: there exists a constant $\Lip_x L$ such that for any }$h \in \mathcal{M}$ \emph{holds }$$\Lip_x \left( Lh \right) \leq \Lip_x L \times  \Lip_x h.$$
\emph{Moreover, if}  $h(b,\cdot) \in C^l$ \emph{for any} $b \in \Td$, \emph{then} $Lh$ \emph{has the same smoothness as well and}
\begin{equation}
\left|\left|(Lh)^{(l)}\right|\right|_C \leq C_k(L) \left|\left|h^{(l)}\right|\right|_C.
\end{equation}

\begin{proof}
Using the explicit formula for the solution \eqref{eq_solution_of_homological_equation_formal_series}, as well as bounds \eqref{eq_obvious_2} and \eqref{eq_obvious}, we have:
\begin{multline}
\sup_{x,y \in [0,1]} \left| \frac{L h (b,x)-Lh(b,y)}{x-y}\right|=\sup_{x,y \in [0,1]} \left|
\sum_{k=0}^{\infty} P_k(b) \frac{h \circ F_0^k(b,x)-h \circ F_0^k(b,y)}{x-y}
\right| \leq\\
\leq
\sup_{x,y \in [0,1]}
\sum_{k=0}^{\infty}
P_k(b)\frac{\Lip_x h |\Pi_k(b) x -\Pi_k(b) y|}{|x-y|} =
\Lip_x h \frac{D}{1-q^2}.
\end{multline}

The bounds for the derivatives are obtained analgously by differentiating term by term the series \eqref{eq_solution_of_homological_equation_formal_series}:
\begin{align}
(Lh)^{(l)} =-\sum_{k=0}^{\infty} P_k(b) \Pi_k^l(b) h^{(l)} \circ F_0^k.
\end{align}
Therefore,
\begin{equation}
\left|\left|(Lh)^{(l)}\right|\right|_C \leq \frac{D}{1-q^{l+1}}\left|\left|h^{(l)}\right|\right|_C .
\end{equation}
\end{proof}

\textbf{Proposition 6.}
\emph{There exist polynomials} $T_3(\varepsilon)$\emph{ and} $T_4^0(\varepsilon)$ \emph{of degrees respectively} $3$ \emph{and} $4$ \emph{such that} $T_4^0(0)=0$ \emph{and for any} $h \in N$\emph{ holds}
\begin{equation}
\Lip_{x,\varepsilon} [\Phi h] \leq T_3(\varepsilon)+T_4^0(\varepsilon) A_{\Lip}
\end{equation}

Proof of Proposition \ref{prop_lip_calc}:
The proof of this proposition deals with an expression for $\Lip_{x,\varepsilon} \Phi h$ which is given by
\begin{multline}
\sup_{x,y \in [0,\varepsilon]}\frac{\left|
(1+xh_b(x))^2 Q(b, x+\bar{h}_b(x))-(1+yh_b(y))^2 Q(b, y+\bar{h}_b(y))
\right|}{|x-y|}
\end{multline}
Since in this proposition the base coordinate $b$ is fixed and $x$ is changing we will permit to ourselves not to write the $b$ index and just suppose that $Q(x)=Q(b, x+\bar{h}_b(x))$ as well as $h(x)=h_b(x)$.
The bound is a triangle inequality formula:
\begin{multline}\label{eq_8}
\Lip_{x,\varepsilon} \Phi h \leq \supe \left|
\frac{Q(x) -Q(y)}{x-y}
\right|+2 \supe \left|
\frac{x h(x) Q(x) -y h(y) Q(y)}{x-y}
\right|+\\+
\supe \left|
\frac{x^2 h(x) Q(x)-y^2 h(y) Q(y)}{x-y}
\right| \leq \Lip_x Q (1+\Lip_{x,\varepsilon} \bar{h})+\\ + 2 \supe \left|
\frac{x h(x) \left(Q(x)-Q(y)\right)}{x-y}
\right|
+ 2 \supe \left|
\frac{x Q(y) (h(x)-h(y))}{x-y}
\right|
+2 \sup_{y \in [0,\varepsilon]}\left|
Q(y) h(y)
\right|+\\+\supe \left|
\frac{x^2 h(x) \left(Q(x)-Q(y)\right)}{x-y}
\right|+
\supe
\left|
\frac{Q(y) \left[  x^2\left(h(x)-h(y)\right)+(x^2-y^2)h(y)\right]}{x-y}
\right| \leq \\
\leq
\Lip_x Q (1+\Lip_{x,\varepsilon} \bar{h})+2 \varepsilon A_C \Lip_x Q \Lip_{x,\varepsilon} \bar{h}+ 2 \varepsilon ||Q||_C A_{\Lip} +\\+
 2 ||Q||_C A_C + \varepsilon^2 A_C \Lip_x Q \Lip_{x,\varepsilon} \bar{h}+||Q||_C (\varepsilon^2 A_{\Lip}+ 2 \varepsilon A_C)
\end{multline}
Let us note that
\begin{multline}\label{eq_13}
\Lip_{x,\varepsilon} \bar{h}=\supe \left| \frac{x^2 h(x)-y^2 h(y)}{x-y} \right| \leq \supe \left|\frac{x^2 (h(x)-h(y))}{x-y}\right|+\supe \left| \frac{h(y) (x^2-y^2)}{x-y}\right| \leq\\\leq A_{\Lip} \varepsilon^2 + A_C 2 \varepsilon
\end{multline}
After substitution of $\Lip_{x,\varepsilon} \bar{h}$ in \eqref{eq_8} by \eqref{eq_13} we have the result with
\begin{align*}
T_3(\varepsilon)=2A_C^2 \Lip_x Q \varepsilon^3 + 4 A_C^2 \Lip_x Q  \varepsilon^2+2 A_C (\Lip_x Q +||Q||_C)\varepsilon+ \Lip_x Q + 2 ||Q||_C A_C \\
T_4^0(\varepsilon)=\Lip_x Q A_C \varepsilon^4 + 2 A_C \Lip_x Q \varepsilon^3 + \Lip_x Q \varepsilon^2 + 2 ||Q||_C \varepsilon
\end{align*}

\hfill $\Box$

Now let us prove the analogous proposition for the H\"older norm of the operator $\Phi$:

\textbf{Proposition 7.}
\emph{If} $h \in \mathcal{H^{\alpha}}_{\varepsilon}$ \emph{then }$\Phi h \in \mathcal{H^{\alpha}}_{\varepsilon}$ \emph{as well. And, moreover, for} $h \in N$, \emph{there exist polynomials} $\tilde{T}_2(\varepsilon)$ \emph{and }$\tilde{T}_4^0(\varepsilon), \tilde{T}_4^0(\varepsilon)(0)=0$ \emph{of degrees} $2$ \emph{and }$4$ \emph{correspondingly such that}
\begin{equation}
\left|\left|\
\Phi h
\right|
\right|_{[\alpha], \varepsilon} \leq \tilde{T}_4^0(\varepsilon) A_{\alpha}+ \tilde{T}_2(\varepsilon)
\end{equation}

\begin{proof}

To estimate H\"older norm of the shift operator, we need some more triangle inequalities.
% Here, in comparison to Proposition \ref{prop_lip_calc}, we are working with the base points, so for this proof the following notation will be useful:
%$Q_{b}=Q(b, x+h(b,x))$ as well as $h_b=h(b,x)$.

\begin{multline}
|\Phi h(b_1, x)-\Phi h (b_2,x)|=|(1+x h_{b_1})^2 Q_{b_1} - (1+xh_{b_2})^2 Q_{b_2}| \leq \\
\leq |Q_{b_1}-Q_{b_2}|+ 2 x \left|h_{b_1} Q_{b_1}-h_{b_2} Q_{b_2}\right|+ x^2 \left| h_{b_1}^2 Q_{b_1} - h_{b_2}^2 Q_{b_2} \right| \leq \\
\leq \left|
Q(b_1, x+\bar{h}_{b_1})-Q(b_1, x+\bar{h}_{b_2})
\right|+ \left|
Q(b_1, x+\bar{h}_{b_2})-Q(b_2, x+\bar{h}_{b_2})
\right|+ \\+ 2 \varepsilon \left|
h_{b_1} \left( Q(b_1, x+\bar{h}_{b_1})-Q(b_1, x+\bar{h}_{b_2})\right)
\right|+ 2 \varepsilon \left|
h_{b_1} \left(
Q(b_1, x+\bar{h}_{b_2})
 \right. \right. -\\ \left. \left.
-Q(b_2, x+\bar{h}_{b_2})
\right)
\right|
+ 2 \varepsilon \left|
Q_{b_2} (h_{b_1}-h_{b_2})
\right|+\varepsilon^2 h_{b_1}^2 \left|
Q(b_1, x+\bar{h}_{b_1}) - Q(b_1, x+\bar{h}_{b_2})
\right|+\\+ \varepsilon^2 h_{b_1}^2 \left|
Q(b_1, x+\bar{h}_{b_2})-Q(b_2, x+\bar{h}_{b_2})
\right|+\varepsilon^2 |Q_{b_2}| \left|
h_{b_1}^2-h_{b_2}^2
\right| \leq \\
\leq
||b_1-b_2||^{\alpha} \left(
\tilde{T}_4^0(\varepsilon) A_{\alpha}+\tilde{T}_2(\varepsilon)
\right)
\end{multline}
where
\begin{equation}
\tilde{T_4}^0 (\varepsilon)=\Lip_x Q \varepsilon^2 + 2 \varepsilon^3 A \Lip_x Q + 2 \varepsilon ||Q||_C + \varepsilon^4 A^2 \Lip_Q + 2 ||Q||_C A \varepsilon^2
\end{equation}
and
\begin{equation}
\tilde{T}_2(\varepsilon)=C_Q+ 2 \varepsilon A C_Q+ \varepsilon^2 A^2 C_Q
\end{equation}
\end{proof}

All the propositions stated above are proven. This completes the proof of our main result -- Theorem \ref{thm:new:2}.

Now we are ready to prove that the conjugacy is smooth in fiber variable, and H\"older with its derivatives in base variables.

Proof of Theorem \ref{thm:new:3}:

The proof of a smooth version of Theorem \ref{thm:new:2} is analogous to the proof of the latter theorem.
Here we will give a sketch of the proof: we will only show that the conjugacy $H$ is $(k-2)$-- smooth with respect to the fiber variable. The proof of the fact that its fiber derivatives are now H\"older on $b$ is analogous to the proof of the H\"older property fot the function $H$ itself and we don't give it here.

 The idea is to change the space $\mathcal{N}$ in an appropriate way. For some constants $A_0, \ldots, A_l>0$ and $\varkappa_0, \ldots, \varkappa_l>0$ let us define the space

\begin{equation}
\mathcal{N}_l:=\left\{
h(\cdot, b) \in C^l([0,\varepsilon]) : ||h||_C \leq A_0, \ldots , \left|\left| h^{(l)}\right|\right|_C \leq A_l
\right\}
\end{equation}
 with the norm
 \begin{equation}\label{eq_norm}
 ||h||_*=\varkappa_0 ||h||_C + \ldots + \varkappa_l ||h^{(l)}||_C.
 \end{equation}\

 We have now to prove the analogues of Lemmas \ref{lem_homological}, \ref{lem_contraction} and \ref{lem_shift} above, and then follow the argument in Theorem \ref{thm:new:2}. The homological and shift operators will stay the same although the functional spaces in which they act will be smaller, and the metric will be not continuous but a smooth one.

\textbf{Lemma 1} (smooth case)
Operator $L$ is bounded in the norm \eqref{eq_norm}.
\begin{proof}
\begin{multline}
||L h||_{*} = \sum_{j=0}^l \varkappa_j || (L h)^{(j)}||_C \leq \sum_{j=0}^l \frac{D \varkappa_j}{1-q^{j+1}}||h^{(j)}||_C
\leq\frac{D}{1-q} \sum_{j=0}^{\infty} \varkappa_j ||h^{(j)}||_C = \frac{D}{1-q}||h||_{*}
\end{multline}
\end{proof}

% For Lemma \ref{lem_shift} to hold, it is sufficient to have $A_0 \geq A_C, A_1 \geq A_{\Lip}$

For the space $\mathcal{N}_l$ to map to itself by $L \Phi$, we should choose constants $A_0, A_1, \ldots , A_l$ appropriately. For $L \Phi$ to be contracting in the space, we should appropriately choose $\varkappa_0, \ldots, \varkappa_l$. Let us show that these two choices can be made without complications and the analogues of Lemmas \ref{lem_contraction} and \ref{lem_shift} hold.

In what concerns the operator $\Phi$, we will use its presentation  \eqref{eq_shift_operator_normalized_form} and calculate the derivatives for $k=0, \ldots, l:$ by the Leibnitz rule:

\begin{equation}
(\Phi h)^{(k)}=\sum_{j=0}^{k} C_k^j ((1+x h(b,x))^2) ^{(j)}Q^{(k-j)}(b, x+\bar{h})
\end{equation}

The explicit form of the right-hand side is not as important as a fact that it can be written as a sum of polynomials in derivatives of $h, \bar{h}$ and $Q$. Indeed, there exist polynomials $\tau_0, \ldots, \tau_l$ and $\sigma_0, \ldots, \sigma_l$ such that

\begin{align}\label{eq_14}
(\Phi h)^{(k)}=\sum_{j=0}^{k} C_k^j \tau_j(x,h, \ldots, h^{(j)}) \sigma_j \left(x, \bar{h}, \ldots, \bar{h}^{(k-j)}, Q\left(b, x+\bar{h}), \ldots,
\right.\right.
 \left. \left.
 Q^{(k-j)}(b,x+\bar{h}
 \right)\right)
\end{align}

We will estimate the continuous norm of  the right-hand side of \eqref{eq_14} in $\mathbb{T}^d \times [0,\varepsilon]$. So we will have that for some polynomials $T_j$ and $S_j$ there is a bound
\begin{align}
||L\Phi h^{(k)}||_{C, \varepsilon} \leq \sum_{j=0}^k C_k^j T_j(\varepsilon, A_0, \ldots, A_j) S_j \left(\varepsilon, A_0, \ldots, A_{k-j}, ||Q||_C, \ldots, \right. \left.
||Q^{(k-j)}||_C\right)
\end{align}

Note that the coefficient in front of $A_k$ in this expression is a polynomial that has no free term. Indeed, $A_k$ comes only from $T_k$ or $S_0$: in both of the cases $A_k$ is multiplied by at least one $\varepsilon$. Hence we need to find $A_0, \ldots A_l$ such that $l+1$ equations hold for some polynomials $P_k, P_k^0, P_k^0(0)=0$:

\begin{equation}\label{eq:alej}
P_k^0(\varepsilon) A_k+P_k(\varepsilon) C(A_0, \ldots, A_{k-1}) \leq A_k, k=0, \ldots, l
\end{equation}

First we take $\varepsilon$ such that all polynomials $P_k^0(\varepsilon)<1$. Then we satisfy the equations \eqref{eq:alej} one by one, starting from $k=0$ by choosing $A_k$ one by one, starting with $A_0$ and by increasing the index.

Now we have to prove that operator $L \Phi$ is contracting in the space $\mathcal{N}$ if  $\varkappa_0, \ldots \varkappa_l$ are properly chosen.

One can show that
\begin{multline}
||L \Phi h - L \Phi g||_{*} \leq \\ \frac{D}{1-q} \sum_{j=0}^l \varkappa_j ||
\sum_{k=0}^{j} C_j^k \left(
\tau_k(x,h, \ldots h^{(k)}) \sigma_k \left(x, \bar{h}, \ldots, \bar{h}^{(j-k)}, Q(b,x+h), \ldots, \right. \right. \\\left. \left. Q^{(j-k)}(b,x+h)\right)-\right.\left.
\tau_k\left(x,g, \ldots g^{(k)}) \sigma_k(x, \bar{g}, \ldots, \bar{g}^{(j-k)}, Q(b,x+g), \ldots,\right.\right. \\
\left.\left.  Q^{(j-k)}(b,x+g)\right)
\right)
||_{C, \varepsilon} \leq
\sum_{j=0}^l ||h^{(k)}-g^{(k)}||_{C, \varepsilon} \left(
\varkappa_k U_k^0(\varepsilon)+\sum_{j={k+1}}^l U_j(\varepsilon) \varkappa_j
\right)
\end{multline}
for some polynomials $U_j, U_j^0$. For the right-hand side to be less than $\xi ||h-g||_{*}$ for some $\xi<1$ the following system should be satisfied:

\begin{align*}
U_0^0 (\varepsilon)+ \frac{\varkappa_1}{\varkappa_0} U_0(\varepsilon) + \ldots \frac{\varkappa_l}{\varkappa_0} U_0(\varepsilon) \leq \xi\\
U_1^0 (\varepsilon)+ \frac{\varkappa_2}{\varkappa_1} U_0(\varepsilon) + \ldots \frac{\varkappa_l}{\varkappa_1} U_0(\varepsilon) \leq \xi \\
\ldots \\
U_{l-1}^0 (\varepsilon) + \frac{\varkappa_l}{\varkappa_{l-1}} U_{l-1}(\varepsilon) \leq \xi\\
U_l^0(\varepsilon) \leq \xi
\end{align*}

One can choose $\varepsilon$ in such a way that $U^0_k(\varepsilon)<\xi$. Then, the last inequality in the list is true, by taking any $\varkappa_l$ and $\varkappa_{l-1}$ big enough, we satisfy the before-last inequality and we proceed in staisfying these inequalities from the last one till the first one.

So, we obtain a contracting operator. We haven provent that in the space $\mathcal{N}_l$ of functions defined in a neighborhood of the base with a metric chosen appropriately, there is a contracting operator $L \Phi$. Its fixed point is the needed conjugacy which will be sufficiently smooth on the fiber variable $x$.

\hfill $\Box$
\section{Acknowledgements}
We would like to thank Ilya Schurov and Stas Minkov for their attentive reading of the preliminary versions of this article and their remarks on the presentation. Olga Romaskevich is supported by UMPA ENS Lyon (UMR 5669 CNRS), the LABEX MILYON (ANR-10-LABX-0070) of Universit\'{e} de Lyon, within the program "Investissements d'Avenir" (ANR-11-IDEX-0007) operated by the French National Research Agency (ANR) as well as by the French-Russian Poncelet laboratory (UMI 2615 of CNRS and Independent University of Moscow). Both authors are supported by RFBR project 13-01-00969-a.


\begin{thebibliography}{99}
\bibitem{program} A.S.Gorodetskii, Yu.S. Ilyashenko \emph{Some new robust properties of invariant sets and attractors of dynamical systems}, Functional Analysis and its Applications, 33, pp.16-30, (1999) (Russian)

A.S.Gorodetskii, Yu.S. Ilyashenko Functional Analysis and its Applications, 33, 95-105 (English translation)
\bibitem{horseshoe}  A.S.Gorodetskii, Yu.S. Ilyashenko \emph{Some properties of skew products over a horseshoe and a solenoid} Tr.Mat.Inst. Steklova, 231, pp. 96 -- 118, 2000, (Russian)

A.S.Gorodetskii, Yu.S. Ilyashenko Din.Sist., Avtom. i Beskon.Gruppy pp.96-118, 2000

A.S.Gorodetskii, Yu.S. Ilyashenko Proc. Steklov Inst.Math. 231, 90-112, 2000 (Engl.transl.)
\bibitem{four}  A. Gorodetski, Yu. Ilyashenko, V. Kleptsyn, M. Nalski, \emph{Non-removable zero Lyapunov exponents}, Functional Analysis and its Applications, 39 (2005), no. 1, pp.27--38
\bibitem{G} A. S. Gorodetskii, \emph{Regularity of central leaves of partially hyperbolic sets and applications}. (Russian) Izv. Ross. Akad. Nauk Ser. Mat. 70 (2006), no. 6, 19--44; translation in Izv. Math. 70 (2006), no. 6, 1093--1116
\bibitem{Gu}  Guysinsky, \emph{ The theory of non-stationary normal forms},
 Ergodic Theory Dynam. Systems  22  (2002),  no. 3, 845--862.
\bibitem{GK}  M. Guysinsky, A.  Katok, \emph{ Normal forms and invariant geometric structures for dynamical systems
 with invariant contracting foliations},
 Math. Res. Lett.  5  (1998),  no. 1-2, 149--163. 
\bibitem{Hartman} P. Hartman, \emph{A lemma in the theory of structural stability of differential equations}, Proc. A.M.S. 11 (1960), no. 4, pp. 610–620.
\bibitem{HPS} M. W. Hirsch, C. C. Pugh, M. Shub, \emph{Invariant manifolds} (Lecture Notes in Mathematics, vol. 583), 1977, ii+149
\bibitem{Skew} Yu. Ilyashenko, \emph{Thick attractors of step skew products}, Regular Chaotic Dyn., 15 (2010), 328--334
\bibitem{Thick} Yu. Ilyashenko, \emph{Thick attractors of boundary preserving diffeomorphisms}, Indagationes Mathematicae, Volume 22, Issues 3--4, December 2011, Pages 257--314
\bibitem{Intermingled} Yu. S. Ilyashenko, \emph{Diffeomorphisms with Intermingled Attracting Basins}, Funkts. Anal. Prilozh., 42:4 (2008), 60--71
\bibitem{inv} Yu.Ilyashenko, A.Negut, \emph{Invisible parts of attractors}, Nonlinearity, 23, pp. 1199-1219, 2010
\bibitem{IN} Yu. Ilyashenko, A. Negut \emph{H\"{o}lder properties of perturbed skew products and Fubini regained}, Nonlinearity 25 (2012) 2377–2399
\bibitem{Nonlocal} Yu. Ilyashenko, Weigu Li \emph{Nonlocal bifurcations}, American Mathematica Society, 1998
\bibitem{KH} A. Katok and B. Hasselblatt, \emph{Introduction to the Modern Theory of Dynamical Systems}, Cambridge Uni.Press, 1994
\bibitem{kv} V. Kleptsyn, D.Volk \emph{Nonwandering sets of interval skew products}, Nonlinearity, 27:1595--1601, 2014
\bibitem{urkud}  Yu. G. Kudryashov, \emph{Bony Attractors}, Funkts. Anal. Prilozh., 44:3 (2010), 73–76
\bibitem{Milnor} J.Milnor, \emph{Fubini Foiled: Katok’s Paradoxical Example in Measure Theory},  The Mathematical Intelligencer, Spring 1997, Volume 19, Issue 2, pp 30--32
\bibitem{SW} M.Shub, A.Wilkinson, \emph{Pathological foliations and removable zero exponents},Inventiones Math. 139 (2000), pp. 495-508
\bibitem{Sternberg}S. Sternberg, \emph{On the structure of local homeomorphisms of Euclidian }$n$\emph{-space}, II. Amer. J. Math \textbf{80}, 623--631 (1958)
\bibitem{Revisited} C.Pugh, M.Shub, A.Wilkinson \emph{H\"{o}lder foliations, revisited}, Journal of Modern Dynamics, 6 (2012) 835--908.
%\bibitem{KH} A. Katok and B. Hasselblatt, \emph{Introduction to the Modern Theory of Dynamical Systems}, Cambridge Uni.
%Press, 1994
\end{thebibliography}
\end{document}